\documentclass[11pt]{article}

\usepackage{booktabs}
\usepackage{relsize}
\usepackage{soul}
\usepackage{graphicx}
\usepackage{pgfplots}
\pgfplotsset{compat=1.18}
\usepgfplotslibrary{fillbetween}
\definecolor{mycolor1}{RGB}{31,114,255}
\definecolor{mycolor2}{RGB}{44,160,44}
\definecolor{mycolor3}{RGB}{255,127,14}
\usepackage{nameref}
\usepackage[shortlabels,inline]{enumitem}
\usepackage{empheq}
\usepackage[shortlabels]{enumitem}
\setlist[enumerate]{nosep}
\usepackage[doc]{optional}
\usepackage{xcolor}
\usepackage[colorlinks=true,
            linkcolor=refkey,
            urlcolor=lblue,
            citecolor=red]{hyperref}
\usepackage{float}
\usepackage{soul}
\usepackage{graphicx}
\definecolor{labelkey}{rgb}{0,0.08,0.45}
\definecolor{refkey}{rgb}{0,0.6,0.0}
\definecolor{Brown}{rgb}{0.45,0.0,0.05}
\definecolor{lime}{rgb}{0.00,0.8,0.0}
\definecolor{lblue}{rgb}{0.5,0.5,0.99}

\usepackage{mathpazo}

\colorlet{hlcyan}{cyan!30}

\usepackage{stmaryrd}
\usepackage{amssymb}

\makeatletter
\def\namedlabel#1#2{\begingroup
   \def\@currentlabel{#2}%
   \label{#1}\endgroup
}
\makeatother

\newcommand{\seppfour}{\setlength{\itemsep}{-4pt}}

\newcommand{\R}[1]{\ensuremath{{\operatorname{R}}_%
{#1}}}

\newcommand{\menge}[2]{\big\{{#1}~\big |~{#2}\big\}}

\newcommand{\fenv}[1]%
{\ensuremath{\,\overrightarrow{\operatorname{env}}_{#1}}}
\newcommand{\benv}[1]%
{\ensuremath{\,\overleftarrow{\operatorname{env}}_{#1}}}

\newcommand{\RR}{\ensuremath{\mathbb R}}

{\begin{list}{}{%
\settowidth{\labelwidth}{\textrm{#1~}}%
\setlength{\leftmargin}{\labelwidth+\labelsep}}}
{\end{list}}
\usepackage{amsthm}
\usepackage{thmtools}
\usepackage[capitalize,nameinlink]{cleveref}
\crefname{lemma}{Lemma}{Lemmas}
\crefname{equation}{}{equations}
\crefname{figure}{Figure}{Figures}
\crefname{chapter}{Appendix}{chapters}
\crefname{item}{}{items}
\crefname{enumi}{}{}
\crefname{enumii}{}{}
\newtheorem{theorem}{Theorem}[section]
\newtheorem{lemma}[theorem]{Lemma}

\newtheorem{corollary}[theorem]{Corollary}

\newtheorem{proposition}[theorem]{Proposition}

\newtheorem{ex}[theorem]{Example}
\newtheorem{fact}[theorem]{Fact}
\newtheorem{remark}[theorem]{Remark}

\providecommand{\RR}{\mathbb{R}}

\providecommand{\sign}{\operatorname{sign}}

\providecommand{\gra}{\operatorname{gra}}

\providecommand{\R}{{ R}}

\providecommand{\RR}{\mathbb{R}}

\providecommand{\C}{\mathcal{C}}

\definecolor{myblue}{rgb}{.8, .8, 1}
  \newcommand*\mybluebox[1]{%
    \colorbox{myblue}{\hspace{1em}#1\hspace{1em}}}

\allowdisplaybreaks 

\newcommand{\px}{\overline{x}}
\newcommand{\py}{\overline{y}}
\usepackage{tikz}
\usetikzlibrary{shapes.misc}
\tikzset{cross/.style={cross out, draw, 
         minimum size=2*(#1-\pgflinewidth), 
         inner sep=0pt, outer sep=0pt}}
\usepackage[dvipsnames]{xcolor}
\usepackage{subcaption}
\usepackage{pgfplots}

\begin{document}

\title{\textsc{
Projection onto the parabola}}
\author{
Francisco J. Arag\'on-Artacho\thanks{
                 Mathematics, Universidad de Alicante, Spain.
                 E-mail: \texttt{francisco.aragon@ua.es}},~
Heinz H.\ Bauschke\thanks{
Mathematics, University
of British Columbia,
Kelowna, B.C.\ V1V~1V7, Canada. E-mail:
\texttt{heinz.bauschke@ubc.ca}.},~
         and C\'esar L\'opez-Pastor\thanks{
                 Mathematics, Universidad de Alicante, Spain.
                 E-mail: \texttt{cesar.lopez@ua.es}.}
                 }

\date{\today}

\maketitle

\vskip 8mm

\begin{abstract} 
In this note, we provide explicit expressions for the projections onto the graph of a quadratic polynomial. The projections are obtained by examining the critical points of the associated quartic polynomial, that is, the roots of the cubic polynomial defining its derivative. We also focus on the case where the point we project lies on the vertical line defined by the parabola. Lastly, an explicit formula for the projection onto a higher dimensional parabola is derived.
\end{abstract}

{\small

}

\section{Introduction}\label{introduction}

Given a general quadratic
\begin{empheq}[box=\mybluebox]{equation*}
s(x) := \alpha x^2 + \beta x + \gamma,
\quad\text{where}\;\;
\alpha \in\RR\setminus\{0\}, \beta \in\RR, \gamma\in\RR,
\end{empheq}
the purpose of this note is to find an explicit formula for the projection onto $S$, the graph (not the epigraph) of 
this quadratic:
\begin{empheq}[box=\mybluebox]{equation*}
S := \gra(s) := \menge{(x,\alpha x^2 + \beta x + \gamma)}{x\in\RR}.
\end{empheq}
Thus, by fixing an arbitrary point that does not belong to the graph
\begin{empheq}[box=\mybluebox]{equation*}
(\px,\py)\in \RR^2\setminus S, 
\quad \text{so }\py\neq s(\px),
\end{empheq}
we are interested in solving the following optimization problem:
\begin{equation}\label{problem}\tag{P}
    \begin{aligned}
    \text{Min }\;&\|(x,\alpha x^2 + \beta x+\gamma)-(\px,\py)\|\\
    \text{s.t. }\;&x\in\RR.
\end{aligned}
\end{equation}
This is equivalent to minimize the quartic 
\begin{align}
Q(x) 
&:= \|(x,\alpha x^2 + \beta x+\gamma)-(\px,\py)\|^2 \nonumber\\
&=
(x-\px)^2 + (\alpha x^2 + \beta x+\gamma-\py)^2 \nonumber\\
&=
\alpha^2 x^4 + \big(2\alpha\beta\big)x^3
+ \big(1+\beta^2 + 2\alpha(\gamma-\py)\big)x^2 
+ \big(2(\beta(\gamma-\py)-\px)\big)x
+ (\gamma-\py)^2 + \px^2.\label{e:Q}
\end{align}
Using the optimality conditions, our goal is to compute the critical points of $Q$ and determine from them its global minima. Since the derivative of a quartic is a cubic polynomial, our problem turns into finding the roots of a cubic. These computations rely on the work \cite{arxiv1}, where the authors provide explicit expressions for these roots. Its main results are summarized in the next section, where we also provide a useful lemma for finding the extrema of a quartic whose derivative has three distinct real roots.

A closed-form expression for computing the projection onto the parabola $4\mathfrak{c}y=x^2$ was recently presented in \cite{chou}. In our notation, this corresponds to $\alpha=1/(4\mathfrak{c})$ and $\beta=\gamma=0$, but this specific form can be assumed without loss of generality by translating the graph. 
Nevertheless, as we point out in \cref{remark}, the formula obtained in~\cite{chou} applies only to certain points in the space, particularly those for which the projection is single-valued. Another related result is~\cite[Theorem~7.1]{arxiv1}, where the authors provide a formula for the projection onto the epigraph of the parabola $h(\mathbf{x})=\alpha\|\mathbf{x}\|^2$, with $\mathbf{x}\in\RR^n$ and $\alpha>0$. In \cref{sec:higher}, we derive a closed-form expression for the projection onto the graph of this parabola, which thus includes points in the interior of the epigraph where the projection is multi-valued. Therefore, the analysis in this note is more comprehensive, covering and extending the corresponding results in \cite{arxiv1} and \cite{chou}.

\section{Auxiliary results}\label{sec:facts}

This section contains some relevant results that will allow us to derive our main theorem on the projections onto a parabola. Let us consider the general cubic polynomial
\begin{empheq}[box=\mybluebox]{equation}
\label{e:gencubic}
f(x) := ax^3+bx^2+cx+d, 
\quad\text{where}\;\;
a,b,c,d\;\text{are in}\;\RR\;\text{and}\; a>0.
\end{empheq}
Note that $f''(x)=6ax+2b$ has exactly one zero, namely
\begin{equation}
\label{e:defx0}
x_0 := \frac{-b}{3a}.
\end{equation}
The change of variables
\begin{equation}
\label{e:charvar}
x = z+x_0 
\end{equation}
leads to the well known depressed cubic
\begin{equation}
\label{e:defpq}
g(z) := z^3 + pz+q, \;\;\text{where}\;\; p := \frac{3ac-b^2}{3a^2}
\;\;\text{and}\;\; q := \frac{27a^2d+2b^3-9abc}{27a^3}.
\end{equation}
Here $ag(z)=f(x)=f(z+x_0)$, so 
the roots of $g$ are precisely those of $f$ translated by $x_0$; i.e.,
\begin{equation*}
\text{$x$ is a root of $f$ $\Leftrightarrow$ $x-x_0$ is a root of $g$.}
\end{equation*}
Hence, to obtain the roots of $f$, all we need to do is find the roots of $g$ and then add $x_0$ to them. Because the change of variables \cref{e:charvar} is linear,
the multiplicity of the roots is preserved. The following result from~\cite{arxiv1} provides explicit expressions of the roots of a general cubic.

\begin{fact}[{\cite[Corollary 3.2]{arxiv1}}] \label{c:genroots} Consider the cubic~\cref{e:gencubic}, recall \cref{e:defx0} and \cref{e:defpq}, and
set  $\Delta := (p/3)^3+(q/2)^2$.
Then exactly one of the following holds:
\begin{enumerate}
\item 
\label{c:genroots1}
\fbox{$p=0$ or $\Delta>0$}\,: 
Then $f$ has exactly one real root and it is given by
\begin{equation}
x_0+ \sqrt[\mathlarger 3]{\frac{-q}{2}+ \sqrt{\Delta}}
+
\sqrt[\mathlarger 3]{\frac{-q}{2}-\sqrt{\Delta}}. 
\end{equation}
\item 
\label{c:genroots2}
\fbox{$p<0$ and $\Delta=0$}\,: 
Then $f$ has two distinct real roots: The simple root is
\begin{equation*}
x_\text{s}:=x_0+\frac{3q}{p} = x_0+2\sqrt[\mathlarger 3]{\frac{-q}{2}}
\end{equation*}
and the double root is
\begin{equation*}
x_\text{d}:=x_0-\frac{3q}{2p} = x_0 -\sqrt[\mathlarger 3]{\frac{-q}{2}}
\end{equation*}
\item 
\label{c:genroots3}
\fbox{$\Delta<0$}\,:
Then $f$ has exactly three real (simple) roots  $r_0,r_1,r_2$, where
\begin{equation}\label{e:theta}
r_k := x_0+2(-p/3)^{1/2}\cos\Big(\frac{\theta+2k\pi}{3} \Big), 
\quad\text{}\;\;
\theta := \arccos \frac{-q/2}{(-p/3)^{3/2}},
\end{equation}
and $r_1<r_2<r_0$. 
\end{enumerate}
\end{fact}

The following lemma will allow us to identify the global minimizers of $Q$ in~\cref{e:Q} when its derivative has three distinct real roots.

\begin{lemma}\label{l:Alicante1}
Let $Q(x)$ be a real quartic polynomial with leading coefficient 
$ax^4$ and $a>0$. Suppose that its derivative is
\begin{equation*}
f(x) := Q'(x) = 4a(x-r_1)(x-r_2)(x-r_3),
\end{equation*}
with distinct real roots ordered as $r_1<r_2<r_3$. 
Then $r_1$ and $r_3$ are the only local minimizers of $Q$, and $r_2$ is the only local maximizer of $Q$. Further, the following holds:
\begin{enumerate}
\item 
If $2r_2 - r_1-r_3=0$, then both $r_1$ and $r_3$ are the only global minimizers of $Q$.
\item
If $2r_2 - r_1-r_3<0$, then $r_3$ is the unique global minimizer of $Q$.
\item 
If $2r_2 -r_1-r_3>0$, then $r_1$ is the unique global minimizer of $Q$.
\end{enumerate}
Put differently, the set of global minimizers of $Q$ is the subset of roots 
drawn from $\{r_1,r_3\}$ that are farthest away from $r_2$.
\end{lemma}
\begin{proof}
Because $a>0$, we have: $f$ is negative on $\left]-\infty,r_1\right[$,
$f$ is positive on $\left]r_1,r_2\right[$, 
$f$ is negative on~$\left]r_2,r_3\right[$, 
and $f$ is positive on $\left]r_3,\infty\right[$. 
Since $Q'=f$, it follows that 
$r_1$ and $r_3$ are local minimizers of~$Q$, and $r_2$ is a local maximizers of $Q$,
and there are no other critical points of $Q$. 
We have 
\begin{equation*}
Q(r_2)-Q(r_1) = \int_{r_1}^{r_2}f(x)\,dx= \frac{1}{3} \, a {\left(r_{1} + r_{2} - 2 \, r_{3}\right)} {\left(r_{1} - r_{2}\right)}^{3}.
\end{equation*}
The assumption $r_1<r_2<r_3$ implies $Q(r_1)<Q(r_2)$ and similarly $Q(r_3)<Q(r_2)$. Likewise,
\begin{equation*}
Q(r_3)-Q(r_1)
=
\frac{1}{3} \, a {\left(r_{1} - 2 \, r_{2} + r_{3}\right)} {\left(r_{1} - r_{3}\right)}^{3},
\end{equation*}
which yields the rest of the result. 
\end{proof}

\section{Computing the Projection}

Since $Q$ in~\cref{e:Q} is a quartic polynomial, we can compute its critical points using the derivative. We obtain the cubic
\begin{equation}\label{e:f}
f(x) := Q'(x) = ax^3 + bx^2 + cx+ d,
\end{equation}
where
\begin{align*}
a &:= 4 \, \alpha^{2} > 0, \\
b &:= 6 \, \alpha \beta, \\
c &:= 2 \, \beta^{2} - 4 \, \alpha \py + 4 \, \alpha \gamma + 2, \\
d &:= -2 \, \beta \py + 2 \, \beta \gamma - 2 \, \px.
\end{align*}
Recalling \cref{l:Alicante1}, we are interested in the roots of $f$, as they provide the 
points to investigate to solve the original problem. 
Now, let us define the usual key quantities from \cref{sec:facts} associated with $f$: 
\begin{equation*}
x_0 := \frac{-b}{3a} = -\frac{\beta}{2 \, \alpha}
\end{equation*}
and 
\begin{equation*}
\label{e:240622a}
p := \frac{3ac-b^2}{3a^2}
\;\;\text{and}\;\; q := \frac{27a^2d+2b^3-9abc}{27a^3}. 
\end{equation*}
Thankfully, this simplifies to
\begin{equation}\label{e:p}
p = -\frac{\beta^{2} + 4 \, \alpha \py - 4 \, \alpha \gamma - 2}{4 \, \alpha^{2}}\;\;\text{and}\;\; q = -\frac{2 \, \alpha \px + \beta}{4 \, \alpha^{3}}.
\end{equation}

\begin{theorem}\label{t:main}
    Let $s(x)=\alpha x^2+\beta x+\gamma$ be a quadratic function with $\alpha\neq 0$. Let $\Delta := (p/3)^3+(q/2)^2$, where $p$ and $q$ are given by \cref{e:p}. Then, for any $(\px,\py)\notin S:=\gra(s)$, the projection onto $S$ is computed as
    \[P_S(\px,\py)=\begin{cases}
        (x_1,s(x_1)) & \text{if }p=0\text{ or }\Delta>0,\\
        (x_2,s(x_2)) & \text{if }p<0\text{ and }\Delta=0,\\
        \{(x^+_3,s(x^+_3)),(x^-_3,s(x^-_3))\} &\text{if }\Delta<0\text{ and }\px=-\frac{\beta}{2\alpha},\\
        (x_4,s(x_4)) &\text{if }\Delta<0\text{ and }q<0,\\
        (x_5,s(x_5)) &\text{if }\Delta<0\text{ and }q>0,\\
    \end{cases}\]
    where $x_1$, $x_2$, $x_3^+$, $x_3^-$, $x_4$, $x_5$ are defined as
    \begin{align*}
        x_1&:=-\frac{\beta}{2\alpha}+\sqrt[\mathlarger 3]{\frac{2\alpha\px+\beta}{8\alpha^3}-\sqrt{\Delta}}+\sqrt[\mathlarger 3]{\frac{2\alpha\px+\beta}{8\alpha^3}+ \sqrt{\Delta}},\\
        x_2&:=
-\frac{\beta}{2 \, \alpha} + 
\frac{3 \, {\left(2 \, \alpha \px + \beta\right)}}{{\left(\beta^{2} + 4 \, \alpha \py - 4 \, \alpha \gamma - 2\right)} \alpha},\\
x_3^\pm&:=-\frac{\beta}{2\alpha}\pm\sqrt{-p},\\
x_4&:=-\frac{\beta}{2\alpha}+2\sqrt{\frac{-p}{3}}\cos\big(\theta/3\big),\\
     x_5&:=-\frac{\beta}{2\alpha}+2\sqrt{\frac{-p}{3}}\cos\Big(\frac{\theta+2\pi}3\Big),
    \end{align*}
and $\theta := \arccos \frac{-q/2}{(-p/3)^{3/2}}$.
\end{theorem}

\begin{proof}
Since the set of minimizers of $Q(x)$ in \cref{e:Q} is contained in the set of roots of the cubic $f(x)$ in~\cref{e:f}, the results introduced in \cref{sec:facts} are key to find the solutions of the optimization problem and, hence, the projections onto $S$. Thus, we will go through each case of \cref{c:genroots} to obtain an explicit formulation for the projection.

\noindent\fbox{Case (i): $p=0$ or $\Delta>0$.} By \cref{c:genroots}, $f$ has exactly one real root, which is given by
\begin{equation}
x_1 := -\frac{\beta}{2\alpha}+u_-+u_+,\quad
\text{where}\;\;
u_{\pm} := \sqrt[\mathlarger 3]{\frac{2\,\alpha\px+\beta}{8\alpha^3}\pm \sqrt{\Delta}}. 
\end{equation}
Thus, $Q$ has exactly one minimizer, as there are no other critical points, and the projection onto $S$ is given by
\begin{equation*}
P_S(\px,\py) = (x_1,s(x_1)).
\end{equation*}

\noindent\fbox{Case (ii): $p<0$ and $\Delta=0$.} In this case, $f$ has two distinct real roots: a double root $x_\text{d} := x_0-\frac{3q}{2p}$ and one simple root
\[x_\text{s}:= x_0+\frac{3q}{p}= 
-\frac{\beta}{2 \, \alpha} + 
\frac{3 \, {\left(2 \, \alpha \px + \beta\right)}}{{\left(\beta^{2} + 4 \, \alpha \py - 4 \, \alpha \gamma - 2\right)} \alpha}.\]
Recalling~\cref{e:f}, we have that $f(x)=4\alpha^2(x-x_\text{d})^2(x-x_\text{s})$. As the sign of $f$ only changes at $x_\text{s}$, the point $x_d$ cannot even be a local minimizer of $Q$. Therefore, $P_S(\px,\py)=(x_s,s(x_s))$.

\noindent\fbox{Case (iii): $\Delta<0$.} This is the case where $f$ has three simple real roots
$r_1<r_2<r_0$, whose expressions are given by  
\begin{equation}\label{e:rk}
r_k := x_0+2\sqrt{\frac{-p}{3}}\cos\Big(\frac{\theta+2k\pi}{3} \Big),
\end{equation}
with
\begin{equation}\label{e:theta2}
\theta := \arccos \frac{-q/2}{(-p/3)^{3/2}}=\arccos \frac{3^{3/2}\operatorname{sign}(\alpha)(2\,\alpha\px+\beta)}{(\beta^{2} + 4 \, \alpha \py - 4 \, \alpha \gamma - 2)^{3/2}}\in\left]0,\pi\right[. 
\end{equation}
We represent these three roots in \cref{fig:roots}.

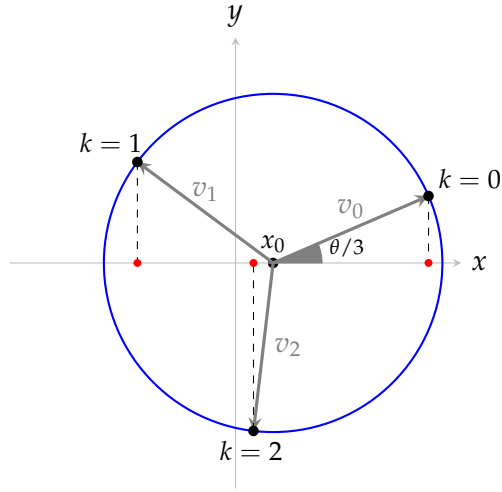
\begin{figure}[ht]
\centering
\begin{tikzpicture}[>=stealth]
    \def\x{0.5}
    \def\y{0}
    
    \coordinate (x0) at (\x,\y);
    \fill (x0) circle (2pt) node[above] {\small $x_0$};
    
    \def\R{2.25}
    \def\thetaval{70}
    
    \draw[->, gray!50] (-3,0) -- (3,0) node[right, black] {$x$};
    \draw[->, gray!50] (0,-3) -- (0,3) node[above, black] {$y$};
    
    \draw[thick, blue] (x0) circle (\R);
    \draw[fill,gray] (x0) -- ({\x+0.65},\y) arc [radius=0.65, start angle=0, end angle={\thetaval/3}] -- cycle;
    \node at ([shift={(x0)}]\thetaval/6:\R/2.25) {\scriptsize $\theta/3$};

    \foreach \k in {0,1} {
        \pgfmathsetmacro{\ang}{(\thetaval + \k*360)/3}
        
        \coordinate (P\k) at ([shift={(x0)}]\ang:\R);

        \draw[very thick, ->,gray] (x0) -- (P\k) node [midway,above]  {$v_\k$};)
        \draw[dashed] (P\k) -- ({\x+\R*cos(\ang)}, 0);
        \fill[red] ({\x+\R*cos(\ang)}, 0) circle (1.5pt);

        \fill (P\k) circle (2pt);
        
        \node[anchor=180+\ang] at (P\k) {\small $k=\k$};
    }
    \foreach \k in {2} {
        \pgfmathsetmacro{\ang}{(\thetaval + \k*360)/3}
        
        \coordinate (P\k) at ([shift={(x0)}]\ang:\R);

        \draw[very thick, ->,gray] (x0) -- (P\k) node [midway,below, right]  {$v_\k$};)
        \draw[dashed] (P\k) -- ({\x+\R*cos(\ang)}, 0);
        \fill[red] ({\x+\R*cos(\ang)}, 0) circle (1.5pt);

        \fill (P\k) circle (2pt);
        
        \node[anchor=180+\ang] at (P\k) {\small $k=\k$};
    }
\end{tikzpicture}
\caption{Geometric representation of the real roots $r_k$ (red dots) given in \cref{e:rk}}
\label{fig:roots}
\end{figure}

By \cref{l:Alicante1}, $r_0$ and $r_1$ are two local minimizers of $Q$,
and $r_2$ is a local maximizer. Moreover, the global minimizers of $Q$ correspond to the roots that are farthest away 
from the middle root $r_2$. Hence, we care about the sign of $2r_2 - r_0-r_1$.
Observing that the vectors $v_0$, $v_1$ and $v_2$ depicted in~\cref{fig:roots} add up to $0$, the sum of their $x$-coordinates satisfies 
$$\cos\Big(\frac{\theta}{3}\Big)+\cos\Big(\frac{\theta+2\pi}{3}\Big)+\cos\Big(\frac{\theta+4\pi}{3}\Big)=0.$$ 
Then, we obtain the following simplification:
\[\frac{2r_2 - r_0-r_1}{2(-p/3)^{1/2}}=3\cos\Big(\frac{\theta+4\pi}{3}\Big)=3\sin\Big(\frac{\theta-\pi/2}{3}\Big).\]
Hence, we need to check whether the above sine is zero, negative or positive. Recalling that $0<\theta<\pi$, checking the sign of the sine is equivalent to the following conditions, respectively: $\theta=\pi/2$, $0<\theta<\pi/2$ or $\pi/2<\theta<\pi$.

\noindent$\bullet$ The case $\theta=\pi/2$ is equivalent to $\frac{-q/2}{(-p/3)^{3/2}}=0$, i.e., $q=0$. Under this condition, the global minimizers of $Q$ are given by
\begin{align*}
    r_0&=x_0+2\sqrt{-p/3}\cos\Big(\frac{\pi}{6} \Big)=-\frac{\beta}{2\alpha}+\sqrt{-p}=:x_3^+,\\
    r_1&=x_0+2\sqrt{-p/3}\cos\Big(\frac{\pi/2+2\pi}{3} \Big)=
    -\frac{\beta}{2\alpha}-\sqrt{-p}=:x_3^-.
\end{align*}

\noindent$\bullet$ The condition $0<\theta<\pi/2$ is equivalent to $0 < (-q/2)/(-p/3)^{3/2}<1$. However, by the assumptions of this case, $\Delta<0$ and $p<0$, which directly implies that $(-q/2)/(-p/3)^{3/2}<1$ always holds. Hence, the chain of inequalities $0<\theta<\pi/2$ is characterized by just $0 < (-q/2)/(-p/3)^{3/2}$, namely, $q<0$. In this case, the unique global minimizer is
\[r_0=-\frac{\beta}{2\alpha}+2\sqrt{\frac{-p}{3}}\cos\Big(\frac{\theta}{3} \Big)=:x_4.\]

\noindent$\bullet$ Lastly, the condition {$\pi/2<\theta<\pi$} is equivalent to $0 > (-q/2)/(-p/3)^{3/2}>-1$ which, by similar arguments as before, is characterized by $q>0$. In this case, the unique global minimizer is
\[r_1=-\frac{\beta}{2\alpha}+2\sqrt{\frac{-p}{3}}\cos\Big(\frac{\theta+2\pi}{3} \Big)=:x_5.\]

Summing up, when $\Delta<0$, the projection onto $S$ is given by 
$$P_S(\px,\py)=\begin{cases}
    \{(x_3^+,s(x_3^+)),(x_3^-,s(x_3^-))\} & \text{if } q=0,\\
    (x_4,s(x_4)) &\text{if } q<0,\\
    (x_5,s(x_5)) &\text{if }q>0.
\end{cases}$$
This concludes the proof.
\end{proof}

The computation of the projection given in \cref{t:main} is divided into five disjoint cases, which follow those provided by \cref{c:genroots} and \cref{l:Alicante1}. However, these cases can be summarized into only three. 

\begin{corollary}\label{c:main}
    Let $s(x)=\alpha x^2+\beta x+\gamma$ be a quadratic function with $\alpha\neq 0$. Let $\Delta := (p/3)^3+(q/2)^2$, where $p$ and $q$ are given by \cref{e:p}. Then, for any $(\px,\py)\notin S:=\gra(s)$, the projection onto $S$ is given by
    \[P_S(\px,\py)=\begin{cases}
        (z_1,s(z_1)) & \text{if }\Delta\geq 0,\\
        \{(z^+_2,s(z^+_2)),(z^-_2,s(z^-_2))\} &\text{if }\Delta<0\text{ and }\px=-\frac{\beta}{2\alpha},\\
        (z_3,s(z_3)) &\text{if }\Delta<0\text{ and }\px\neq-\frac{\beta}{2\alpha}.
    \end{cases}\]
    where $z_1$, $z_2^+$, $z_2^-$, and $z_3$ are defined as
    \begin{align*}
        z_1&:=-\frac{\beta}{2\alpha}+\sqrt[\mathlarger 3]{\frac{2\alpha\px+\beta}{8\alpha^3}-\sqrt{\Delta}}+\sqrt[\mathlarger 3]{\frac{2\alpha\px+\beta}{8\alpha^3}+ \sqrt{\Delta}},\\
        z_2^\pm&:=-\frac{\beta}{2\alpha}\pm\sqrt{-p},\\
        z_3&:=-\frac{\beta}{2\alpha}-2\sign(q)\sqrt{-p/3}\cos(\vartheta/3), 
    \end{align*}
and $\vartheta:=\arccos\left|\frac{q/2}{(-p/3)^{3/2}}\right|$.
\end{corollary}

\begin{proof}
This is a consequence of combining the cases for the projection in~\cref{t:main}. Indeed, observe first that when $\Delta=0$, we have that $p\leq0$. When $p=0$, the projection is given by $(x_1,s(x_1))$. Otherwise, we are in case (ii) of \cref{c:genroots}, and 
since $\Delta=0$,
$$x_2=x_\text{s}=x_0+\frac{3q}{p}= x_0+2\sqrt[\mathlarger 3]{\frac{-q}{2}}=\sqrt[\mathlarger 3]{\frac{-q}{2}+ \sqrt{\Delta}}+\sqrt[\mathlarger 3]{\frac{-q}{2}-\sqrt{\Delta}}=x_1.$$ 
Hence, both cases can be subsumed into the combined region $\Delta\geq 0$, in which the projection is given by $(z_1,s(z_1))$, with $z_1:=x_1$. 

Consider now the case where $\Delta<0$ and $q\neq 0$. Let us show that the expression of $x_4$ and $x_5$ can be combined as $z_3:=x_0-2\sign(q)\sqrt{-p/3}\cos(\vartheta/3)$, where $\vartheta:=\arccos\left|\frac{q/2}{(-p/3)^{3/2}}\right|$. To prove this, let $w:=\frac{-q/2}{(-p/3)^{3/2}}$. On the one hand, when $q>0$, one has 
\begin{align*}
    \cos\Big(\frac{\vartheta}{3}\Big)&=\cos\Big(\frac{\arccos(-w)}{3}\Big)=\cos\left(\frac{\pi-\arccos(w)}{3}\right)
    =-\cos\Big(\frac{\theta+2\pi}{3}\Big),
\end{align*}
whence, $z_3=x_0+2\sqrt{-p/3}\cos\big((\theta+2\pi)/3\big)=x_5$.
On the other hand, if $q<0$, the value of $\theta$ in \cref{t:main} coincides with that of $\vartheta$, so $z_3=x_4$ (since $\sign(q)=-1$). 
\end{proof}

\cref{fig:regions} illustrates for the parabola $s(x)=2x^2+x-1$ the three regions derived from the previous analysis.

\begin{figure}[ht]
\centering
\begin{tikzpicture}

    \pgfmathsetmacro{\A}{2}
    \pgfmathsetmacro{\B}{1}
    \pgfmathsetmacro{\C}{-1}
    \pgfmathsetmacro{\Z}{-\B/(2*\A)}
    \pgfmathsetmacro{\xmin}{-1.7}
    \pgfmathsetmacro{\xmax}{1.3}
    \pgfmathsetmacro{\ymin}{-1.5}
    \pgfmathsetmacro{\ymax}{0.6}
    
    \pgfmathsetmacro{\discTerm}{4*\A*\C - \B*\B + 2}
    \pgfmathsetmacro{\Ystart}{\discTerm / (4*\A)}

    \pgfmathdeclarefunction{s}{1}{\pgfmathparse{\A*#1^2 + \B*#1 + \C}}
    \pgfmathdeclarefunction{Delta}{1}{\pgfmathparse{(3*(abs(2*\A*#1 + \B))^(2/3) + \discTerm) / (4*\A)}}

    \begin{axis}[
        xmin=\xmin, xmax=\xmax,
        ymin=\ymin, ymax=\ymax,
        xtick distance=0.4,
        ytick distance=0.4,
        axis lines=middle,
        samples=500, 
        domain=\xmin:\xmax,
        clip=true,
        font=\small,
        legend pos=south east,
        width=0.85\textwidth,
        height=0.6\textwidth,
        legend cell align={left}]


    \addplot[color=mycolor1, line width=2pt, name path=parabola]{s(x)};
    \addlegendentry{$s(x)=2x^2+x-1$}


    \path[name path=ymin_path] (\xmin,\ymin) -- (\xmax,\ymin);
    \path[name path=ymax_path] (\xmin,\ymax) -- (\xmax,\ymax);
    \addplot[draw=none, name path=secondcurve, forget plot] {Delta(x)};


    \addplot[draw=none, color=mycolor3, opacity=0.35,legend image post style={draw=mycolor3!60!white,line width=1pt, opacity=1, fill opacity=0.35}] fill between [of=secondcurve and ymin_path, soft clip={domain=\xmin:\xmax}];
    
    \addlegendentry{$\Delta\geq0$}

    
    \addplot[color=red, line width=1pt] coordinates {(\Z, \Ystart+0.025) (\Z, \ymax)};
    \addlegendentry{$\Delta<0$ and $\px=-\frac{\beta}{2\alpha}$}

    
    \addplot[draw=none, fill, opacity=0.35,color=yellow] fill between[of=secondcurve and ymax_path,soft clip={domain=\xmin:\xmax}];
    \addlegendentry{$\Delta <0$ and $\px\neq-\frac{\beta}{2\alpha}$}

    \addplot[color=red, line width=1pt, forget plot] coordinates {(\Z, \Ystart+0.025) (\Z, \ymax)};
    \addplot[color=mycolor3!60!white, line width=1pt, name path=secondcurve, forget plot] {Delta(x)};


    \addplot[color=mycolor1, line width=2pt, name path=parabola, forget plot]{s(x)};

\end{axis}
\end{tikzpicture}
\caption{The three regions defining the projection for a particular parabola}
\label{fig:regions}
\end{figure}
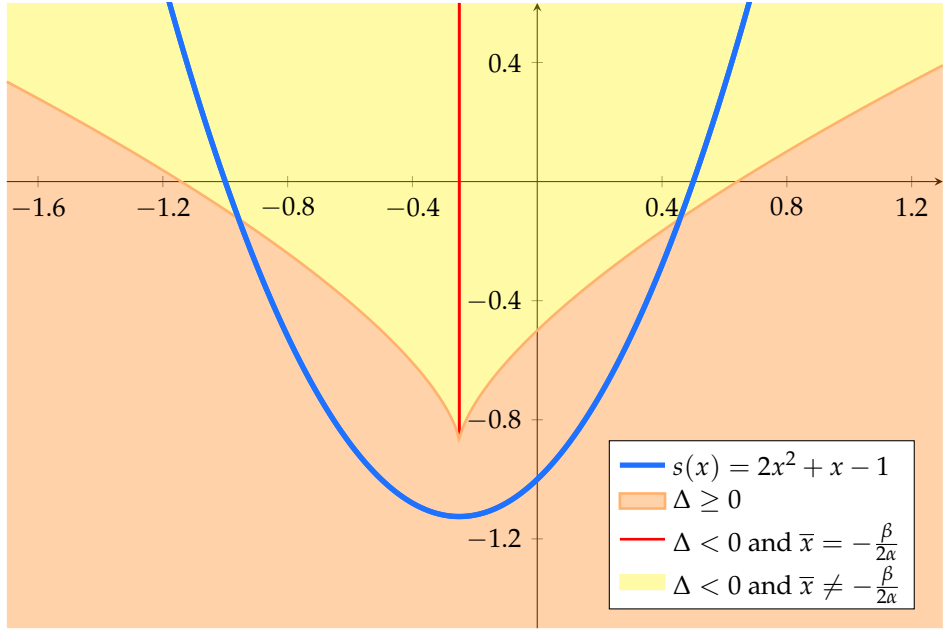

\begin{remark}[Connection with \cite{chou}]\label{remark}
    Let us consider the parabola $4\mathfrak{c}y=x^2$, with $\mathfrak{c}\neq0$. If $(\px,\py)\in\RR^2$ is the point to be projected, then by~\cref{e:p} we have $p=4\mathfrak{c}(2\mathfrak{c}-\py)$ and $q=-8\mathfrak{c}^2\px$. The projection given in \cite[Equation (5)]{chou} is obtained by finding the roots of \cite[Equation (4)]{chou}:
    \[x^2=4\mathfrak{c}\Big(\py+\frac{2\mathfrak{c}}{x}(\px-x)\Big).\]
    At first glance, this equation already imposes the condition $x\neq0$, meaning that the $x$-coordinate of the projection point cannot be zero. This removes the case where $\px$ lies on the axis of symmetry, i.e., $\px\neq x_0$, and in particular excludes the case $P_S(\px,\py)=\{(z_2^\pm,s(z_2^\pm))\}$ from the final formulation. Continuing with this assumption, the previous equation turns into the following depressed cubic equation:
    \[x^3+4\mathfrak{c}(2c-\py)x-8\mathfrak{c}^2\px=x^3+px+q=0.\]
    The author claims that this equation has three different roots, one real and two complex conjugates. However, according to \cref{c:genroots}, this only occurs when $\Delta:=(p/3)^3+(q/2)^2>0$. Indeed, this premise is reflected in \cite[Equation (5)]{chou} by the computation $\sqrt{46656\mathfrak{c}^4\px^2+4(3p)^3}=3^3\cdot 2\sqrt{\Delta}$. Consequently, the case $P_S(\px,\py)=(z_3,s(z_3))$ from \cref{c:main} is also omitted.
\end{remark}

\section{The projection of a point lying on the axis of symmetry}

Let us now focus on the case $q=0$. Using \cref{e:p}, this corresponds to the choice
\begin{equation}\label{e:xbar}
    \px=x_0=-\frac{\beta}{2\alpha},
\end{equation}
that is, the point $(\px,\py)$ lies on the vertical line $x=x_0$, which is the $x$-coordinate of the parabola vertex. The $y$-coordinate of the vertex can be easily computed as
\begin{equation}\label{e:s(xbar)}
    s(x_0)=\alpha\Big(\frac{-\beta}{2\alpha}\Big)^2+\beta\Big(\frac{-\beta}{2\alpha}\Big)+\gamma=\gamma-\frac{\beta^2}{4\alpha}.
\end{equation}
Recalling the definition of $\Delta$, under the assumption $q=0$ we get $\Delta=(p/3)^3$. Thus, $\Delta\geq0$ is equivalent to $p\geq0$ and, likewise, $\Delta<0$ is equivalent to $p<0$. By \cref{c:main}, the following conditions can only occur when $q=0$:
\begin{itemize}
    \item If $p\geq0$, then $P_S(x_0,\py)=(z_1,s(z_1))$ is single-valued.
    \item If $p<0$, then $P_S(x_0,\py)=\{(z^+_2,s(z^+_2)),(z^-_2,s(z^-_2))\}$ is multi-valued.
\end{itemize}
We are interested in the value $p=0$, the frontier between single- and multi-valued projections. From \cref{e:p}, we solve for $\py$ to obtain
\begin{equation}\label{e:ybar}
\py=y_0:=\frac{4\alpha\gamma-\beta^2+2}{4\alpha}.
\end{equation}
Finally, we recall that the \emph{focus} of the parabola is the point
\begin{equation}\label{e:focus}
    F:= (F_x,F_y) := \Big(x_0,s(x_0)+\frac1{4\alpha}\Big),
\end{equation}
while the \emph{directrix} is the line $\RR\times\{\gamma -(\beta^2+1)/(4\alpha)\} 
= \RR\times \{s(x_0) - 1/(4\alpha)\}$. 

We illustrate in \cref{fig:2} the projection depending on the position of the point $(\px,\py)$ with respect to $(x_0,y_0)$.
\begin{figure}[ht!]
    \centering
    \label{fig:placeholder}
\begin{subfigure}{0.3\textwidth}
\centering
\begin{tikzpicture}
  \draw[-] (-2,0) -- (2,0);
  \draw[-] (0,-0.5) -- (0,3);
  \draw[domain=-2:2, smooth, variable=\x, mycolor1, line width=1pt] plot ({\x},{0.5*\x*\x+0.5});
  \draw [dashed] (0,1.2) circle [radius=0.7];
  \draw[dashed,lightgray] (0,0.5) -- (0,1.2);
  \draw[fill,OliveGreen] (0,1.5) +(-0.06,-0.06) rectangle +(0.06,0.06);
  \draw[fill,red] (0,1) circle [radius=0.05];
  \draw[fill] (0,.5) circle [radius=0.05];
  
  \draw (0, 1.2) node[cross=3.5pt, line width=1pt] {};
   
\end{tikzpicture}
\caption{$p>0\Leftrightarrow \py<y_0$}
\end{subfigure}
\begin{subfigure}{0.3\textwidth}
\centering
\begin{tikzpicture}
  \draw[-] (-2,0) -- (2,0);
  \draw[-] (0,-0.5) -- (0,3);
  \draw[domain=-2:2, smooth, variable=\x, mycolor1, line width=1pt] plot ({\x},{0.5*\x*\x+0.5});
  \draw[dashed] (0,1.5) circle [radius=1];

  \draw[dashed,lightgray] (0,0.5) -- (0,1.5);
  \draw[fill,OliveGreen] (0,1.5) +(-0.06,-0.06) rectangle +(0.06,0.06);
  \draw[fill,red] (0,1) circle [radius=0.05];
  \draw[fill] (0,.5) circle [radius=0.05];
  \draw (0, 1.5) node[cross=3.5pt, line width=1pt] {};
   
\end{tikzpicture}
\caption{$p=0\Leftrightarrow \py=y_0$}
\end{subfigure}
\begin{subfigure}{0.3\textwidth}
\centering
\begin{tikzpicture}
  \draw[-] (-2,0) -- (2,0);
  \draw[-] (0,-0.5) -- (0,3);
  \draw[domain=-2:2, smooth, variable=\x, mycolor1, line width=1pt] plot ({\x},{0.5*\x*\x+0.5});
  \draw[fill,OliveGreen] (0,1.5) +(-0.06,-0.06) rectangle +(0.06,0.06);
  \draw[fill,red] (0,1) circle [radius=0.05];
  \draw[dashed] (0,2.28125) circle [radius=1.600781];
  \draw[dashed,lightgray] (-1.25,1.28125) -- (0,2.28125) -- (1.25,1.28125);
   
  \draw (0, 2.28125) node[cross=3.5pt, line width=1pt] {};
    
  \draw[fill] (-1.25,1.28125) circle [radius=0.05];
  \draw[fill] (1.25,1.28125) circle [radius=0.05];
\end{tikzpicture}
\caption{$p<0\Leftrightarrow \py>y_0$}
\end{subfigure}
\caption{Projection 
(represented by black dot(s)) onto a parabola (in blue) of a point
$(\px,\py)$ (black cross) lying on the axis of symmetry, depending on its position with respect to the point $(x_0,y_0)$ (green square). The focus $F$ is represented by a red dot, 
and the directrix is the horizontal axis.}\label{fig:2}
\end{figure}
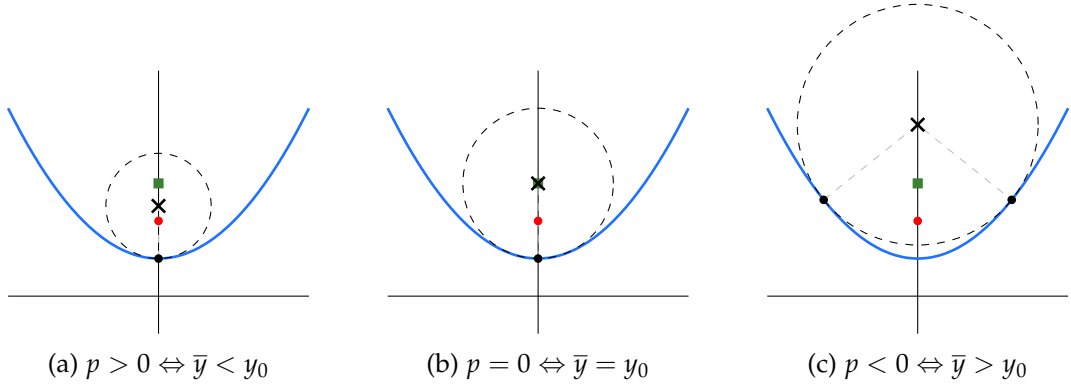

\begin{proposition}
    Let $(x_0,y_0)$ be as in \cref{e:xbar} and \cref{e:ybar}. Then
    \begin{enumerate}
        \item The point $(x_0,y_0)$ is the reflection of the vertex $(x_0,s(x_0))$ across the focus; in particular, $y_0=2F_y-s(x_0)$.
        \item $|y_0-s(x_0)|$ is the radius of curvature of $S$ at $(x_0,y_0)$.
    \end{enumerate}
\end{proposition}

\begin{proof}
    The result is clear for the $x$-component. 
    For the $y$-component, we just gather \cref{e:s(xbar)}, \cref{e:ybar} and \cref{e:focus}.
    \[2F_y-s(x_0)=s(x_0)+\frac1{2\alpha}=\gamma-\frac{\beta^2}{4\alpha}+\frac1{2\alpha}=\frac{4\alpha\gamma-\beta^2+2}{4\alpha}=y_0.\]
    Now, since $S$ is given by the function $s(x)$, we compute the radius of curvature at $x_0$ as
    \[\left|\frac{(1+s'(x_0)^2)^{3/2}}{s''(x_0)}\right|=\left|\frac{(1+0)^{3/2}}{2\alpha}\right|=\frac{1}{2|\alpha|}.\]
    Lastly, observe that $y_0-s(x_0)=2(F_y-s(x_0))=1/(2\alpha)$. 
    This completes the proof. 
\end{proof}

\section{Projection onto a higher dimensional parabola}\label{sec:higher}

Let us consider the function 
\begin{empheq}[box=\mybluebox]{equation*}
t:\RR^n\to\RR: \mathbf{x}\mapsto \alpha\|\mathbf{x}\|^2,\quad\text{where}\;\;
\alpha\neq 0.
\end{empheq}
An expression for the projection onto its graph
\begin{empheq}[box=\mybluebox]{equation*}
T := \gra(t) := \menge{(\mathbf{x},\alpha \|\mathbf{x}\|^2)}{\mathbf{x}\in\RR^n}
\end{empheq}
is provided next.
\begin{theorem}
    Let $t(\mathbf{x})=\alpha \|\mathbf{x}\|^2$  for $\mathbf{x}\in\RR^n$, with $\alpha>0$. Consider any $(\mathbf{\px},\py)\notin T:=\gra(t)$ and let 
    $$\Delta := \frac{(1-2\alpha\py)^3}{216\alpha^6}+\frac{\|\mathbf{\px}\|^2}{16\alpha^4}.$$
    Then the projection of $(\mathbf{\px},\py)$ onto $T$ is given by
    \[P_T(\mathbf{\px},\py)=\begin{cases}
        (\mathbf{0},0) &\text{if }\mathbf{\px}=\mathbf{0}\text{ and } \Delta\geq 0,\\
        \menge{(\mathbf{w},t(\mathbf{w}))}{\|\mathbf{w}\|=\sqrt{(1/\alpha)\py-1/(2\alpha^2)}}&\text{if } \mathbf{\px}=\mathbf{0}\text{ and } \Delta< 0,\\
        (\mathbf{w}_1,t(\mathbf{w}_1)) & \text{if }\mathbf{\px}\neq \mathbf{0}\text{ and }\Delta\geq 0,\\
        (\mathbf{w}_2,t(\mathbf{w}_2)) &\text{if }\mathbf{\px}\neq \mathbf{0}\text{ and }\Delta<0,
    \end{cases}\]
    where $\mathbf{w}_i:=(w_i/\|\mathbf{\px}\|)\mathbf{\px}$, with
    \begin{align*}
        w_1&:=\sqrt[\mathlarger 3]{\frac{\|\mathbf{\px}\|}{4\alpha^2}-\sqrt{\Delta}}+\sqrt[\mathlarger 3]{\frac{\|\mathbf{\px}\|}{4\alpha^2}+ \sqrt{\Delta}},\\
        w_2&:=2\sqrt{\frac{2\alpha\py-1}{6\alpha^2}}\cos(\vartheta/3), 
    \end{align*}
and $\vartheta:=\arccos\big(\frac{3\sqrt{6}}{2}\alpha\|\mathbf{\px}\|/(2\alpha\py-1)^{3/2}\big)$.
\end{theorem}
\begin{proof}
To compute the projection, we must solve the optimization problem
\begin{equation*}
    \begin{aligned}
    \text{Min }\;&\|(\mathbf{x},\alpha \|\mathbf{x}\|^2)-(\mathbf{\px},\py)\|\\
    \text{s.t. }\;&\mathbf{x}\in\RR^n.
\end{aligned}
\end{equation*}
Thus, we are interested in minimizing the function
$$
    g(\textbf{x}):=\|(\mathbf{x},\alpha\|\mathbf{x}\|^2)-(\mathbf{\px},\py)\|^2=\|\mathbf{x}-\mathbf{\px}\|^2+(\alpha\|\mathbf{x}\|^2-\py)^2.
$$
A minimizer $\mathbf{w}\in\RR^n$ of $g$ must satisfy
$$
    \nabla g(\textbf{w})=2(\mathbf{w}-\mathbf{\px})+4\alpha(\alpha\|\mathbf{w}\|^2-\py)\mathbf{w}=\mathbf{0},
$$
that is, 
\begin{equation}\label{e:collinear}
    (2\alpha^2\|\mathbf{w}\|^2+1-2\alpha\py)\mathbf{w}=\mathbf{\px}.
\end{equation} 
Let us consider two cases:

\noindent$\bullet$ If $\mathbf{\px}= \mathbf{0}$, then either $\mathbf{w}=\mathbf{0}$ or $2\alpha^2\|\mathbf{w}\|^2+1-2\alpha\py=0$. The latter case gives the equation $\|\textbf{w}\|^2=(1/\alpha)\py-1/(2\alpha^2)$, which requires $2\alpha\py-1\geq 0$ (or, equivalently, $\Delta\leq 0$). Since it holds that $g(\textbf{w})=\py/\alpha-1/(4\alpha^2)\leq \py^2=g(\mathbf{0})$, the minimum of the function $g$ is attained at any point $\mathbf{w}\in\RR^n$ such that $\|\textbf{w}\|=\sqrt{(1/\alpha)\py-1/(2\alpha^2)}$, whenever $\Delta< 0$, and at $(\mathbf{0},0)$ otherwise.

\noindent$\bullet$ If $\mathbf{\px}\neq 0$, by~\cref{e:collinear} we can write $\mathbf{w}=\mu\frac{\mathbf{\px}}{\|\mathbf{\px}\|}$. Thus, we are interested in finding the value of $\mu$ that minimizes 
$$g\Big(\mu\frac{\mathbf{\px}}{\|\mathbf{\px}\|}\Big)=(\mu-\|\mathbf{\px}\|)^2+(\alpha\mu^2-\py)^2,$$
which corresponds to the problem of finding the projection of the point $(\|\mathbf{\px}\|,\py)\in\RR^2$ onto the parabola $s(x)=\alpha x^2$. This can be computed using \cref{c:main}, from where the expressions follow.
\end{proof}

In \cref{fig:parabola3D} we illustrate the projection onto a parabola for $n=2$. On the left, the black point has a multi-valued projection (the red circle). On the right, the projection consists of a unique (red) point. 

\begin{figure}[ht!]
\hfill
\includegraphics[width=.38\textwidth]{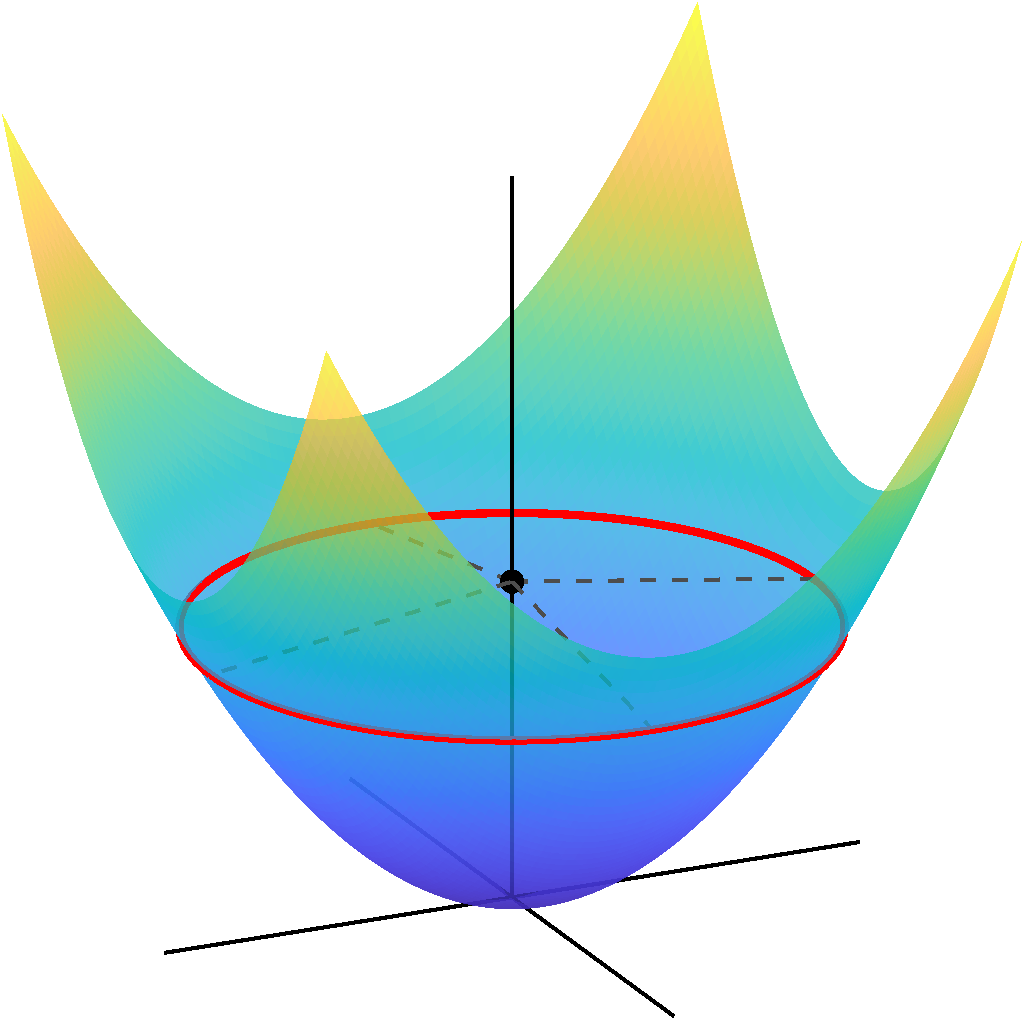}\hfill~
\includegraphics[width=.38\textwidth]{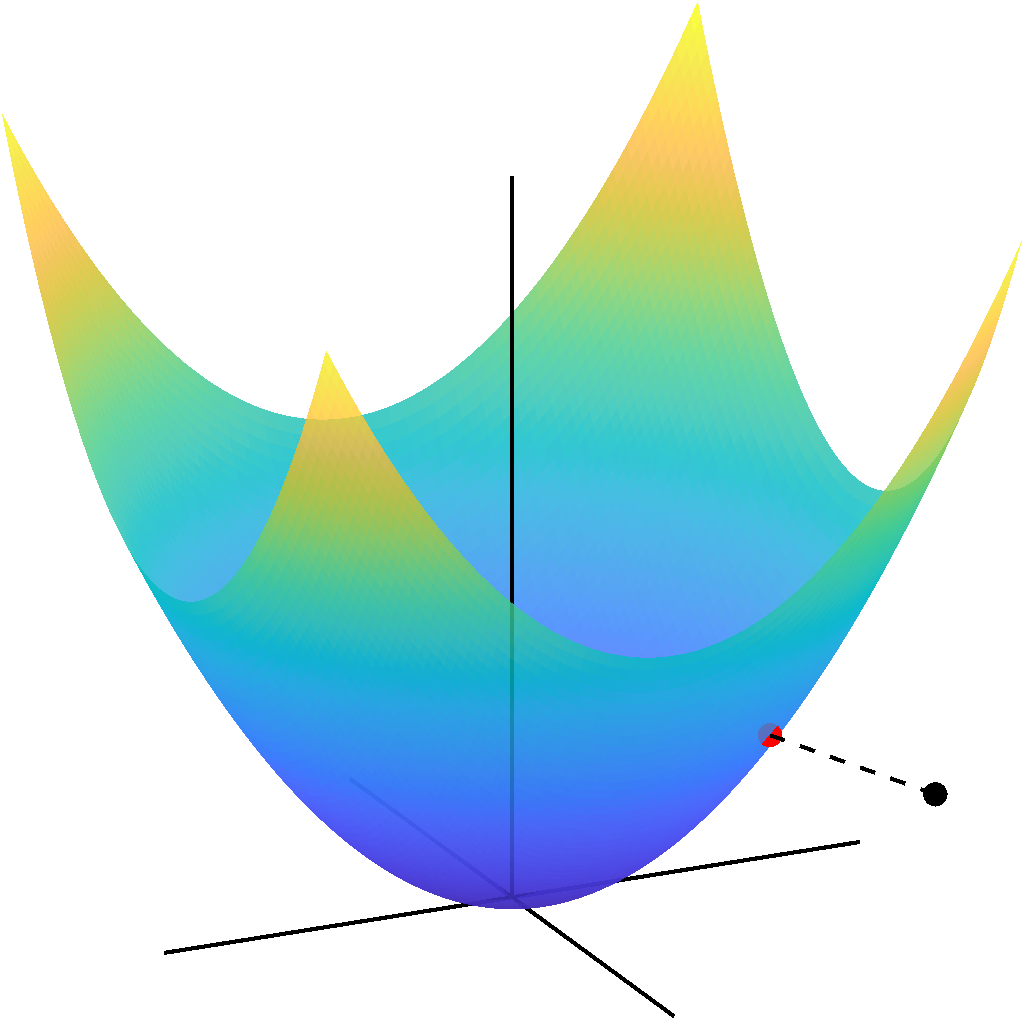}\hfill~
\caption{Multi-valued projection (left) and single-valued projection (right)}
\label{fig:parabola3D}
\end{figure}

\section*{Acknowledgments}
FJAA and CLP were partially supported by Grant PID2022-136399NB-C21 funded by ERDF/EU and by MICIU/AEI/10.13039/501100011033. The research of HHB was supported by the Natural Sciences and Engineering Research Council of Canada. CLP was supported by Grant PREP2022-000118 funded by MICIU/AEI/10.13039/501100011033 and by ``ESF Investing in your future''.

\end{document}